\documentclass[twoside,centertags]{amsart}

\usepackage{amsfonts}
\usepackage{amsthm}
\usepackage{amsmath}
\usepackage{amssymb}
\usepackage[all]{xy}
\usepackage[latin1]{inputenc}

\theoremstyle{plain}
\newtheorem{theorem}{Theorem}[section]
\newtheorem{corollary}[theorem]{Corollary}
\newtheorem{lemma}[theorem]{Lemma}
\newtheorem{proposition}[theorem]{Proposition}
\theoremstyle{definition}
\newtheorem{defi}[theorem]{Definition}
\theoremstyle{remark}
\newtheorem{remark}[theorem]{Remark}

\newcommand{\C}{{\mathcal C}}

\newcommand{\D}{{\mathcal D}}
\newcommand{\M}{{\mathcal M}}
\newcommand{\W}{{\mathcal W}}
\newcommand{\calL}{{\mathcal L}}
\newcommand{\calS}{{\mathcal S}}

\newcommand{\holim}{\mathop{\textrm{\rm holim}}}
\newcommand{\colim}{\mathop{\textrm{\rm colim}}}
\newcommand{\hocolim}{\mathop{\textrm{\rm hocolim}}}

\newcommand{\Ho}{{\text{\rm Ho}}}
\newcommand{\map}{{\text{\rm map}}}
\newcommand{\Map}{{\text{\rm Map}}}

\newcommand{\Cyl}{{\text{\rm Cyl}}}

\newcommand{\sCat}{\text{\bf sCat}}
\newcommand{\Cat}{\text{\bf Cat}}
\newcommand{\wCat}{\text{\bf wCat}}

\SelectTips{cm}{}

\numberwithin{equation}{section}

\begin{document}

\title{The hammock localization preserves homotopies}
\author{Oriol Ravent\'os}
\address{Fakult\"{a}t f\"{u}r Mathematik,
Universit\"{a}t Regensburg,
93040 Regensburg,
Germany}
\email{oriol.raventos-morera@ur.de}
\date{}

\thanks{The author is supported by the grant SFB 1085 ``Higher invariants'' founded by the German Research Foundation, the project CZ.1.07/2.3.00/20.0003 of the Operational Programme Education for Competitiveness of the Ministry
of Education, Youth and Sports of the Czech Republic, the Spanish Ministry
of Science and Innovation under the grant MTM2010-15831 and MTM-2013-42178-P, and the Government of Catalonia under the grant SGR-119-2009.}

\maketitle

%\footnote{2010 {\em Mathematics Subject Classification:\/} Primary 55U35; Secondary 55P60, 18C35.}
%\footnote{{\em Key words:} model category, homotopy function complex, localization, homotopy algebra.}
%\footnote{The author is supported by the project CZ.1.07/2.3.00/20.0003
%of the Operational Programme Education for Competitiveness of the Ministry
%of Education, Youth and Sports of the Czech Republic, the Spanish Ministry
%of Science and Innovation under the grant MTM2010-15831, and the Government of Catalonia under the grant %SGR-119-2009.}

\begin{abstract}
  The hammock localization provides a model for a homotopy function complex in any Quillen model category. We prove that a homotopy between a pair of morphisms induces a homotopy between the maps induced by taking the hammock localization. We describe applications of this fact to the study of homotopy algebras over monads and homotopy idempotent functors. Among other things, we prove that, under Vop{\v e}nka's principle, every homotopy idempotent functor in a cofibrantly generated model category is determined by simplicial orthogonality with respect to a set of morphisms. We also give a new proof of the fact that left Bousfield localizations with respect to a class of morphisms always exist in any left proper combinatorial model category under Vop{\v e}nka's principle.
\end{abstract}

\section{Introduction}

The hammock localization was introduced by Dwyer and Kan in a series of articles \cite{DK80b}, \cite{DK80a} and \cite{DK80c}. Given a category $\C$ with a fixed class of morphisms $\W$, the hammock localization $\calL^H\C$ is a simplicial category such that $\pi_0(\calL^H\C(X,Y))$ is the set of morphisms from $X$ to $Y$ in the category obtained by inverting the morphisms in $\W$ for every pair of objects $X$ and $Y$ in $\C$. In the case that $\C$ is a model category and $\W$ is its class of weak equivalences, then $\pi_0(\calL^H\C (X,Y))$ is in natural bijection with the set of homotopy classes of morphisms $[X,Y]$ and, as a bifunctor, $\calL^H\C(-,-)$ sends weak equivalences to weak homotopy equivalences. Hence, $\calL^H\C(-,-)$ defines a homotopy function complex on $\C$. Moreover, if $\C$ is a simplicial model category, with simplicial mapping space $\Map(-,-)$, then $\calL^H\C(X,Y)\simeq\Map(X^c,Y^f)$, where $X^c$ is a cofibrant replacement of $X$ and $Y^f$ is a fibrant replacement of $Y$.

In Theorem \ref{thm} we prove that $\calL^H\C(-,-)$ sends left or right homotopies to simplicial homotopies. This is applied in Section \ref{sec-homotopical localizations} to study homotopy idempotent functors. We recall that a (coaugmented) homotopy idempotent functor on a model category $\C$ is a functor $L\colon\C\to\C$ together with a natural transformation $\ell\colon 1\to L$ that induces a localization, i.e., a left adjoint of the inclusion of a reflective subcategory, in the homotopy category. An object $X$ is $L$-local if it is weakly equivalent to an object of the form $LY$ for some $Y$, and a morphism $f$ is an $L$-equivalence if $Lf$ is a weak equivalence. We prove in Proposition \ref{prop} that, in any model category, $L$-local objects and $L$-equivalences are simplicially orthogonal with respect to $\calL^H\C(-,-)$. If we assume a certain large cardinal axiom, called Vop{\v e}nka's principle, we prove in Corollary \ref{cor-cofib-gen1} that for each homotopy idempotent functor $(L,\ell)$ in any cofibrantly generated model category, the class of $L$-local objects correspond to the class of objects that are simplicially orthogonal to \emph{just a set} of morphisms. This result extends a previous result in \cite[Theorem~2.3]{CC06} for simplicial combinatorial model categories to all cofibrantly generated model categories. In the same spirit, we extend  in Theorem \ref{cell1} the analogous result for augmented homotopy idempotent functors \cite[Theorem~2.1]{Ch07}.

It was proved in \cite[Theorem~2.3]{RT03} that, under Vop{\v e}nka's principle, left Bousfield localizations with respect to a class of morphisms exist in any left proper combinatorial model category. We give a new proof of this fact in Corollary \ref{Bousf-combinatorial}. The proof can be easily modified to give the analogous result for right Bousfield localizations as we state in Corollary \ref{cell2}. This last result extends a previous result in \cite[Theorem~1.4]{Ch07} for simplicial combinatorial model categories to all combinatorial model categories.
%�can be extended to cofibrantly generated?
%These results, generalize some previous results in \cite{CSS}, \cite{CC06} and \cite{RT03}.

The hammock localization $\calL^H$ can be extended to a functor from the category of small categories with weak equivalences to the category of small simplicial categories
$$
\calL^H\colon\wCat\longrightarrow \sCat
$$
as we make precise in Section \ref{sec-hammocks}. We prove that $\calL^H$ can be extended so as to send natural transformations to simplicial natural transformations \emph{up to homotopy} in Theorem \ref{thm-nat-transf}. Even if this does not make $\calL^H$ a strict $2$-functor, it is already useful for some applications. In Section \ref{sec-ho-alg}, we give an application to the study of homotopy algebras. Roughly, we transfer the property that every homotopy algebra is a homotopy retract of a free algebra to a statement about homotopy function complexes. This result is used in a joint paper of the author with Casacuberta and Tonks \cite{CaRaTo}, in which we study homotopy algebra structures preserved by localizations.

%\noindent {\bf Acknowledgements.}
\subsection*{Acknowledgements} We want to thank Ilias Amrani, John Bourke, Carles Casacuberta, Javier Guti{\' e}rrez, Alexandru Stanculescu, George Raptis and Andrew Tonks for useful discussions during the preparation of this article.

\section{The hammock localization}

The hammock localization defines one model for the homotopy function complex of a model category. It was introduced by Dwyer and Kan in a series of articles \cite{DK80b}, \cite{DK80a} and \cite{DK80c}. We will summarize some of their results following the more recent exposition contained in \cite[Chapter~34]{DHKS04}.

A \emph{category with weak equivalences} is a pair $(\C,\W)$ with $\C$ a category and $\W$ a fixed class of morphisms in $\C$ that contains all identities. The morphisms in $\W$ are called \emph{weak equivalences}. Assume (just for the moment) that $\C$ is small. For every pair of objects $X$ and $Y$ in $\C$, and every odd natural number $n$, we define a category $\calL^H_n\C(X,Y)$ with objects being strings of $n$ morphisms on $\C$ in alternating directions
\begin{equation}\label{zig-zag}
\xymatrix{C_0 & \ar[l]_-{d_0} C_1 \ar[r]^-{d_1} & \dots \ar[r]^-{d_{n-2}} & C_{n-1} & \ar[l]_-{d_{n-1}} C_n, }
\end{equation}
with $X=C_0$, $Y=C_n$ and the arrows pointing to the left being weak equivalences. A~morphism is a commutative diagram of the form
$$
\xymatrix{ C_0 \ar@{=}[d] & \ar[l] C_1 \ar[r] \ar[d] & \dots \ar[r] & C_{n-1} \ar[d] & \ar[l] C_n \ar@{=}[d] \\
 C'_0 & \ar[l] C'_1 \ar[r] & \dots \ar[r] & C'_{n-1} & \ar[l] C'_n. }
$$
The \emph{hammock localization} of $(\C,\W)$ is a simplicial category (meaning simplicially enriched) $\calL^H\C$ with the same objects as $\C$ and, for every pair of objects $X$ and $Y$, a simplicial set
$$
\calL^H\C(X,Y)=\colim_n N \calL^H_n\C(X,Y),
$$
where the sequential colimit (that is also a homotopy colimit) is taken over the nerve of the embedding functors which send an object like (\ref{zig-zag}) in $\calL^H_n\C(X,Y)$ to
\begin{equation}\label{inclusion}
\xymatrix{X=C_0 & \ar[l] C_1 \ar[r]^{id} & C_1 & \ar[l]_{id} C_1 \ar[r] & \dots \ar[r] & C_{n-1} & \ar[l] C_n=Y}
\end{equation}
in $\calL^H_{n+2}\C(X,Y)$. The composition in $\calL^H\C$ is given by concatenation. More precisely, given an object
$$
\xymatrix{X & \ar[l] C_1 \ar[r] & \dots \ar[r] & C_{n-1} & \ar[l]_-{d_{n-1}} Y }
$$
in $\calL^H_n\C(X,Y)$ and an object
$$
\xymatrix{Y & \ar[l]_-{d'_{0}} C'_1 \ar[r] & \dots \ar[r] & C'_{n-1} & \ar[l] Z }
$$
in $\calL^H_n\C(Y,Z)$, their composition is the object
$$
\xymatrix@C=1.1cm{X & \ar[l] C_1 \ar[r] & \dots \ar[r] & C_{n-1} & \ar[l]_-{d_{n-1}\circ d'_0} C'_1 \ar[r] & \dots \ar[r] & C'_{n-1} & \ar[l] Z }
$$
in $\calL^H_{2n-1}\C(X,Z)$. To see that this is well defined, one uses the well known facts that filtered colimits commute with finite limits, that nerves commute with products and that the category of simplicial sets is a cartesian closed model category.
%The identity of $X$ is determined by $[1]\to \calL^H_1(X,X)$ with image
%$$
%\xymatrix{X \ar[r]^{id} \ar@{=}[d] & X \ar@{=}[d]\\ X \ar[r]^{id} & X.}
%$$
%We do not justify that this works wells with the sequential colimit

\begin{remark}
The hammock localization was originally defined using a colimit over all natural numbers, cf.\ \cite{DK80c}. We restrict to odd natural numbers because in this case the morphisms in the extremes are always going backwards and then we do not need to distinguish two cases in every proof. It can be seen that both definitions coincide using a cofinality argument, as proved in \cite[Chapter~34]{DHKS04}. It is also worth mentioning that if $\C$ is a model category and $\W$ is its class of weak equivalences, then $\calL^H\C(X,Y)\simeq\calL^H_3\C(X,Y)$ for every pair of objects $X$ and $Y$, cf.\ ~\cite{DK80c}. Although working with $\calL^H_3$ has certain advantages, for the purposes of this article it is more convenient to work with $\calL^H$ even in the case of model categories.
\end{remark}

\begin{remark}\label{Groth-remark}
We recall that the nerve of a category $\D$ is the simplicial set with $n$\nobreakdash-simplices $(N\D)_n=\Cat([n],\D)$ and this defines a fully faithful functor from $\Cat$ to the category of simplicial sets. We will often use the fact that natural transformations induce homotopies after taking nerves \cite[Section~1, Proposition~2]{Qu73}. In particular, adjunctions induce homotopy equivalences. It is also useful to know that, for any pair of objects $X$ and $Y$, $\calL^H\C(X,Y)$ is weakly equivalent to the nerve of the Grothendieck construction on the diagram defining the sequential colimit, as observed in \cite[Proposition 35.7]{DHKS04}.
\end{remark}

Let $\wCat$ denote the category of small categories with weak equivalences and morphisms being the functors that preserve weak equivalences. Then there is a functor
$$
\calL^H\colon \wCat \longrightarrow \sCat,
$$
where $\sCat$ is the category of small simplicial categories. The image of a functor in $\wCat$ is defined levelwise in each category $\calL^H_n\C(X,Y)$.

Notice, in particular, that for every morphism $f\colon A\to B$ in $\C$ there are induced maps of simplicial sets
$$
f^*\colon \calL^H\C(B,Y)\longrightarrow \calL^H\C(A,Y) \mbox{{\rm \, and \,}} f_*\colon \calL^H\C(X,A)\longrightarrow \calL^H\C(X,B).
$$
To be more precise, $f^*$ is induced by the functors $f^*_n$ that send an object like (\ref{zig-zag})
in $\calL^H_n\C(B,Y)$ to
$$
\xymatrix{A & \ar[l]_{id} A \ar[r]^-f & B & \ar[l] C_1 \ar[r] & \dots \ar[r] & C_{n-1} & \ar[l] Y }
$$
in $\calL^H_{n+2}\C(A,Y)$ for every odd natural number $n$. If $f$ is a weak equivalence, then $f^*$ is a weak homotopy equivalence and a homotopy inverse is given by the functors that send an object like (\ref{zig-zag}) in $\calL^H_n\C(A,Y)$ to
$$
\xymatrix{B & \ar[l]_-{f} A \ar[r]^-{id} & A & \ar[l] C_1 \ar[r] & \dots \ar[r] & C_{n-1} & \ar[l] Y }
$$
in $\calL^H_{n+2}\C(B,Y)$ for every odd natural number $n$. Indeed, if $f$ is a weak equivalence, then $f^*_n$ is an equivalence of categories for each $n$. The map $f_*$ is defined similarly.

If $\C$ is a model category we will let $\W$ be exactly the class of weak equivalences in $\C$. In this case, $\pi_0(\calL^H\C(X,Y))\cong\Ho(\C)(X,Y)$ and $\calL^H\C(X,Y)$ defines a \emph{homotopy function complex} (or \emph{homotopy mapping space}) for $\C$, cf.\ \cite[Chapter~17]{Hi03}.

We would like to apply the hammock localization not only to small categories. This has some technical set theoretical issues that can be nicely handled using the axiomatization of universes. We refer to \cite[Section~32]{DHKS04} for a detailed explanation.

\section{A property of the hammock localization}\label{sec-hammocks}

The following result asserts that the hammock localization respects homotopies. For the basic properties of homotopies in model categories we refer to \cite[Chapter~7]{Hi03}. As usual, by \emph{simplicial homotopy} in a simplicial model category we mean the equivalence relation generated by the strict homotopies.

\begin{theorem}\label{thm}
Let $\C$ be a model category and let $f$ and $g$ be left or right homotopic morphisms in $\C$. Then $f_*$ and $g_*$ are simplicially homotopic maps, and $f^*$ and $g^*$ are simplicially homotopic maps.
\end{theorem}

\begin{proof}
Assume that $f$, $g\colon A\to B$ are left homotopic. Fix a cylinder object
$$
\xymatrix@C=1cm{A\coprod A \ar[r]^-{i_0\coprod i_1} & \Cyl(A) \ar[r]^-p & A}
$$
where $p\circ i_0=p\circ i_1=id$ and $i_0$, $i_1$ and $p$ are weak equivalences. Let $H\colon \Cyl(A)\to B$ be a left homotopy between $f$ and $g$. Thus, $H\circ i_0 =f$ and $H\circ i_1= g$. For every object like (\ref{zig-zag}) in $\calL_n^H\C(X,A)$, the commutative diagram
$$
\xymatrix@C=1cm{ X \ar@{=}[d] & \ar[l] C_1 \ar[r] \ar@{=}[d] & \dots \ar[r] & C_{n-1} \ar@{=}[d] & \ar[l]_-{d_{n-1}} A \ar[d]^{i_0} \ar[r]^-{f} & B \ar@{=}[d] & \ar[l]_{id} B \ar@{=}[d] \\
X & \ar[l] C_1 \ar[r] & \dots \ar[r] & C_{n-1} & \ar[l]_-{d_{n-1}\circ p} \Cyl(A) \ar[r]^-{H} & B & \ar[l]_{id} B \\
 X \ar@{=}[u] & \ar[l] C_1 \ar[r] \ar@{=}[u] & \dots \ar[r] & C_{n-1} \ar@{=}[u] & \ar[l]_-{d_{n-1}} A \ar[u]_{i_1} \ar[r]^-{g} & B \ar@{=}[u] & \ar[l]_{id} B \ar@{=}[u]}
$$
determines a zig-zag of natural transformations
$$
\xymatrix{f_*^n \ar[r]^{\phi^n} & \tilde{H}^n & \ar[l]_{\psi^n} g_*^n}
$$
between functors from $\calL^H\C_{n}(X,A)$ to $\calL^H\C_{n+2}(X,B)$ for each odd natural number $n$, which are compatible with the inclusions
$$
\calL^H\C_n(X,A)\to \calL^H\C_{n+2}(X,A).
$$
Now let $\tilde{H}=\colim_n N\tilde{H}^n$. Since the nerve functor sends natural transformations to simplicial homotopies and $\phi^n$ and $\psi^n$ are compatible with the colimit, we have an induced zig-zag of homotopies of simplicial sets $f_*\simeq \tilde{H}\simeq g_*$.

It $f$ and $g$ are right homotopic, then the statement can be proved similarly using path objects.
\end{proof}

The following result describes the $2$\nobreakdash-categorical properties of the hammock localization functor.

\begin{theorem}\label{thm-nat-transf}
Given a natural transformation $\eta\colon F\to G$ between functors $F, G\colon \C\to\D$ in $\wCat$, there is a homotopy
$\calL^H \eta (X,Y)$ from $\eta_{Y*}\circ \calL^H F(X,Y)$ to $\eta_{X}^*\circ \calL^H G(X,Y)$ for each pair of objects $X$ and $Y$ in $\C$:
$$
\xymatrix@R=.3cm@C=.2cm{\calL^H\C(X,Y)\ar[rr]^{\calL^H F(X,Y)} \ar[dd]_{\calL^H G (X,Y)} & & \calL^H\D(FX,FY) \ar[dd]^{\eta_{Y*}} \ar@{=}[dl] \\ & \calL^H\eta (X,Y) \ar@{=>}[dl] & \\ \calL^H\D(GX,GY)\ar[rr]^{\eta_X^*} & & \calL^H\D(FX,GY).}
$$
\end{theorem}

Notice that, for $\calL^H$ to be a strict $2$\nobreakdash-functor, $\calL^H \eta$ would have to define a simplicially enriched natural transformation, i.e., $\calL^H \eta (X,Y)$ would have to be the identity for each pair of objects $X$ and $Y$. Since the simplicial categories in the image of $\calL^H$ are locally nerves of categories (see Remark \ref{Groth-remark}), we can think of them as being $2$-categories. In this sense, $\calL^H\eta$ in Theorem \ref{thm-nat-transf} will define an oplax natural transformation. Because the oplax natural transformations are the $2$-cells of the oplax-Gray category structure on $2$-$\Cat$, i.e., it is enriched with respect to the oplax-Gray tensor product \cite{Gr74}, we can think of $\calL^H$ as a weak map between oplax-Gray categories.

\begin{remark}\label{inclusion-remark}
In the proof of Theorem \ref{thm-nat-transf} we will need the fact that the inclusion described in (\ref{inclusion}), $\calL^H_n\C(X,Y)\hookrightarrow\calL^H_{n+2}\C(X,Y)$, by inserting two consecutive identity morphisms in $C_1$ is related by a zig-zag of natural transformations to the inclusion defined by inserting two consecutive identity morphisms in $C_i$ for any $0\leq i\leq n$.
\end{remark}

\begin{proof}[Proof of Theorem {\rm \ref{thm-nat-transf}.}]
Fix an object like (\ref{zig-zag}) in $\calL^H_n\C(X,Y)$. The homotopy $\calL^H \eta(X,Y)$ is described by the natural transformation defined by the morphisms
$$
\xymatrix@C=0.8cm{FX \ar@{=}[d] & \ar[l]_{id} FX \ar@{=}[d] \ar[r]^{id} & FX \ar[d]^{\eta_X} & \ar[l]_-{Fd_0} FC_1 \ar[d]^{\eta_{C_1}} \ar[r] & \dots \ar[r] & FC_{n-1} \ar[d]^{\eta_{C_{n-1}}} & \ar[l]_-{Fd_{n-1}} FY \ar[d]^{\eta_Y} \ar[r]^-{\eta_Y} & GY \ar@{=}[d] & \ar[l]_{id} GY \ar@{=}[d] \\ FX & \ar[l]_{id} FX \ar[r]^{\eta_X} & GX & \ar[l]_-{Gd_0} GC_1 \ar[r] & \dots \ar[r] & GC_{n-1} & \ar[l]_-{Gd_{n-1}} GY \ar[r]^{id} & GY & \ar[l]_{id} GY }
$$
in $\calL_{n+4}^H\D(FX,GY)$ for each odd number $n$.
\end{proof}

\section{Homotopy algebras over monads}\label{sec-ho-alg}

Recall that a \emph{monad} on a category $\C$ is a triple
$$
T\colon\C\to\C, \quad \eta\colon 1\to T, \quad \mbox{\rm and} \quad \mu\colon T T\to T
$$
such that $\mu\circ T\mu=\mu\circ \mu T$ and $\mu\circ T\eta=\mu\circ\eta T={\rm id}_T$, cf.\ \cite[Chapter~4]{Bo94II}. Any adjunction $F\colon \C\leftrightarrows\D: G$ induces a monad structure on $G F$. A \emph{$T$-algebra} over a monad $(T,\eta,\mu)$ is a pair $(X,a)$ with $a\colon T X\to X$ a morphism in $\C$ such that $a\circ \eta_X={\rm id}$ and $a\circ\mu_X=a\circ Ta$. There is a category of $T$-algebras $\C^T$, also known as the \emph{Eilenberg--Moore category}, together with an adjunction
\[
F : \C\leftrightarrows \C^T: U
\]
such that $U F = T$. The functor $FX=(TX,\mu_X)$ is the free $T$-algebra functor and $U(X,a)=X$ is the forgetful functor.

Given a monad $(T,\eta,\mu)$ on a model category $\M$ that induces a monad on the homotopy category, a \emph{homotopy $T$\nobreakdash-alg\-ebra} is an object of the Eilenberg--Moore category $\Ho(\M)^T$.
A homotopy $T$\nobreakdash-algebra can be thought of as a fibrant and cofibrant object $X$ in $\M$ equipped with a morphism $a\colon TX\to X$ such that $a\circ\eta_X\simeq{\rm id}$ and $a\circ \mu_X\simeq a\circ Ta$. Homotopy $T$\nobreakdash-algebras do not need to agree (not even up to homotopy) with strict $T$\nobreakdash-algebras in $\M$. For this reason, the usual techniques for studying algebras are not always useful in studying homotopy algebras. Such difficulties arose in the joint work of the author in \cite{CaRaTo}. In particular, the following result was needed.

\begin{theorem}\label{ho-alg}
Let $\M$ be a model category and let $(T,\eta,\mu)$ be a monad on $\M$ preserving weak equivalences. Then $\map(f,X)$ is a homotopy retract of $\map(Tf,X)$ for every homotopy $T$\nobreakdash-algebra~$(X,a)$.
\end{theorem}

\begin{proof}
Choose as a model for $\map(-,-)$ the hammock localization $\calL^H\M(-,-)$. There is a diagram
$$
\xymatrix@C=0.4cm@R=0.2cm{\calL^H\M(f,X) \ar[rrr]^-{\calL^H T(f,X)} \ar[ddd]_-{id} & & & \calL^H\M(Tf,TX) \ar[ddd]^-{(\eta_f)^*} \ar[rrr]^-{a_*}
 & & & \calL^H\M(Tf,X) \ar[ddd]^-{(\eta_f)^*} \\ & & \ar@{<=}[dl] & & & \ar@{=}[dl] & \\ & & & & & & \\
\calL^H\M(f,X) \ar[rrr]^-{(\eta_X)_*} & & & \calL^H\M(f,TX) \ar[rrr]^-{a_*} & & & \calL^H\M(f,X)
}
$$
in which the left square commutes only up to homotopy by Theorem \ref{thm-nat-transf} and the right square commutes by the enriched associativity law in $\calL^H \M$. Now, since $a\circ\eta_X\simeq{\rm id}$ and the hammock localization preserves homotopies by Theorem \ref{thm}, $a_*\circ (\eta_X)_*$ is homotopic to the identity. This tells us that $\map(f,X)$ is indeed a homotopy retract of $\map(Tf,X)$.
\end{proof}

In \cite[Section~9]{CaRaTo}, it is studied the invariance of homotopy $T$-algebras under $f$\nobreakdash-localiza\-tions (see Section \ref{sec-homotopical localizations} for the definition). In particular, Theorem \ref{ho-alg} is used to prove the following statement: If $f$ is a morphism in $\M$ such that the localizations $L_f$ and $L_{Tf}$ exist and $T$ preserves $f$\nobreakdash-equivalences and $Tf$\nobreakdash-equivalences, then
$$
L_f X\simeq L_{Tf} X
$$
for every homotopy $T$\nobreakdash-algebra $X$. This result applies, for instance, in the case that $\M$ is the category of pointed simplicial sets and $T$ is the monad associated to a unital operad. In particular, we can take $T$ to be the infinite symmetric product, $\Omega\Sigma$ or $\Omega^\infty\Sigma^\infty$.

\section{Homotopy idempotent functors}\label{sec-homotopical localizations}

We next define an analogue of the notion of idempotent functor, cf.\ \cite[Section 4.2]{Bo94II}, in the context of model categories following \cite{CSS}.

\begin{defi}\label{ho-idempotent}
Let $\M$ be a model category. A functor
$L\colon\M\to\M$ together with a natural transformation $\ell\colon
1\to L$ is called \emph{(coaugmented) homotopy idempotent} if $L$ sends weak equivalences to weak equivalences and the natural
morphisms $\ell_{LX},L\ell_X\colon LX\to LLX$ are equal in the homotopy category $\Ho(\M)$ and both are
weak equivalences for every object $X$ in~$\M$.
\end{defi}
%if we would put that L sends w.e. between cofibrant objects to w.e. we would have problem for the Ld_1 in %the hammocks.

%In any model category, the fibrant replacement functor is homotopy idempotent. Notice that a homotopy %idempotent functor in a model category induces a (strict) idempotent functor in the homotopy category.

There is a notion of \emph{augmented} homotopy idempotent functor, also called \emph{cellularization}. All results in this section have analogues for the augmented case and the proofs can be easily transferred. At the end of the section, we will state the analogues of the two main results.

Given a homotopy idempotent functor $(L,\ell)$, a morphism $f$ in $\M$ is called an \emph{$L$\nobreakdash-equi\-va\-lence}
if $Lf$ is a weak equivalence, and a fibrant object $X$ in
$\M$ is called \emph{$L$\nobreakdash-local} if $X\simeq LY$ for
some $Y$ in $\M$. The class of $L$\nobreakdash-equivalences and
$L$\nobreakdash-local objects determine each other by orthogonality in the homotopy category
$\Ho(\M)$. This means that a morphism $g\colon X\to Y$ is an $L$-equivalence if and only if the morphism
$$
g^*\colon\xymatrix{[Y,Z] \ar[r]^{\cong} & [X,Z]}
$$
is an isomorphism for every $L$-local object $Z$, and a fibrant object $Z$ is $L$-local if and only if $g^*$ is an isomorphism for all $L$-equivalences $g$, cf.\ \cite[Proposition 2.10]{Ad75}.

We will prove in Proposition \ref{prop} that $L$\nobreakdash-equivalences and $L$\nobreakdash-local objects are also \emph{simplicially orthogonal} in the model category. Let us explain what this means. Fix a homotopy function complex $\map(-,-)$ in a model category $\M$ and let $\mathcal{S}$ be any class of morphisms in $\M$. A fibrant object $X$ in $\M$ is called \emph{$\mathcal{S}$\nobreakdash-local} if, for every morphism $f\colon A\to B$ in $\mathcal{S}$, the induced map of homotopy function complexes
$$
f^*\colon \map(B,X)\longrightarrow \map(A,X)
$$
is a weak homotopy equivalence. We denote by $\mathcal{S}^{\perp}$ the class of $\mathcal{S}$\nobreakdash-local objects and we call it the \emph{simplicial orthogonal complement} of $\mathcal{S}$. Similarly, for any class of objects $\mathcal{D}$ in $\M$, a morphism $f\colon A\to B$ is called a \emph{$\mathcal{D}$\nobreakdash-equivalence} if, for every $X\in\mathcal{D}$, $f^*$ is a weak homotopy equivalence. By an abuse of notation, we also denote by $\mathcal{D}^{\perp}$ the class of $\mathcal{D}$\nobreakdash-equivalences and we call it the \emph{simplicial orthogonal complement} of $\mathcal{D}$.

It is important to notice that these definitions do not depend on the choice of homotopy function complex \cite[Proposition 17.8.2]{Hi03}. We fix $\map(-,-)$ to be $\calL^H\C(-,-)$.

Recall from \cite[Definition 3.3.1]{Hi03} that the left Bousfield localization with respect to a class of morphisms $\calS$ on a model category $\M$ (if it exists) is a new model category structure $L_{\calS}\M$ on the same underlying category $\M$ with the same cofibrations and the weak equivalences being the $\calS^\perp$\nobreakdash-equivalences. In particular, if we consider the fibrant replacement functor in $L_{\calS}\M$, then it defines a homotopy idempotent functor on $\M$. We will show that, if we assume that Vop\v{e}nka's principle holds, then in any cofibrantly generated model category, a homotopy idempotent functor has the same local objects as a left Bousfield localization with respect to a set of morphisms.

%We next recall some well\nobreakdash-known facts about model categories that we will use extensively in %what follows: Any model category is saturated, i.e., a morphism is a weak equivalence if and only if it %becomes an isomorphism in the homotopy category; homotopies become equalities in the homotopy category; and %homotopy equivalences (left, right or simplicial) are, in particular, weak equivalences.
%%Saturated \cite[Theorem 8.3.10]{Hi03}
%%Homotopies=equalites in Ho: easy, stated in \cite[10.5]{DHKS} with the usual notation in the homotopy %category: $H\circ i_0 = f\circ p \circ i_0$ and $H\circ i_1 = f\circ p \circ i_1$. Since $i_0$ and $i_1$ %are isos, $f\circ p=H=g\circ p$. Since $p$ is an iso, $f=g$.
%%homotopyequiv are weak equiv: easy from the two previous points \cite[Proposition 2.10]{Adams} but in %\cite[Theorem 7.8.5]{Hi03} claims that is for cofibrant fibrant objects.

\begin{lemma}\label{lemma-eta}
Let $\M$ be a model category and let $(L,\ell)$ be a homotopy idempotent functor on
$\M$. For every pair of objects $X$ and $Y$,
\begin{enumerate}
\item[{\rm 1.}] the map
$\calL^H L(X,LY)\colon \map(X, LY)\to \map(LX,LLY)$ is a simplicial homotopy equivalence, and

\item[{\rm 2.}] the map
$\ell_X^*\colon \map(LX, LY)\to \map(X,LY)$ is also a simplicial homotopy equivalence.
\end{enumerate}
\end{lemma}

\begin{proof}
For the first part, we let $h\colon \map(LX,LLY)\to\map(X,LY)$ be the map induced by the functors $h^n$ that send an object like (\ref{zig-zag}) in $\calL^H_n\C(LX,LLY)$ to
$$
\xymatrix@C=1.4cm{X & \ar[l]_{id} X \ar[r]^-{\ell_X} & LX & \ar[l] C_1 \ar[r] & \dots \ar[r] & C_{n-1} & \ar[l]_-{d_{n-1}\circ\ell_{LY}} LY}
$$
in $\calL^H_{n+2}\C(X,LY)$ for every odd natural number $n$. The homotopy from the identity (see Remark \ref{inclusion-remark}) to $h\circ \calL^H L(X,L Y)$ is determined by the commutative diagram
$$
\xymatrix@C=1.37cm{X \ar@{=}[d] & \ar[l]_{id} X \ar@{=}[d] \ar[r]^{id} & X \ar[d]_{\ell_X} & \ar[l] C_1 \ar[r] \ar[d]_{\ell_{C_1}} & \dots \ar[r] & C_{n-1} \ar[d]_{\ell_{C_{n-1}}} & \ar[l]_{d_{n-1}} LY \ar@{=}[d] \\ X & \ar[l]_{id} X \ar[r]^{\ell_X} & L X & \ar[l] L C_1 \ar[r] & \dots \ar[r] & L C_{n-1} & \ar[l]_-{ L d_{n-1}\circ \ell_{LY}} LY }
$$
in $\calL^H_{n+2}\C(X,LY)$ for every odd natural number $n$. We will now define a zig-zag of homotopies between the identity and $\calL^H L(X,L Y)\circ h$ induced by a zig-zag of natural transformations
\begin{equation}\label{H-zig-zag}
id^n\longrightarrow \tilde{H}^n \longleftarrow \calL^H_n L\circ h^n
\end{equation}
that are compatible with the inclusions
$$
\calL^H_n\C(LX,LLY)\longrightarrow \calL^H_{n+2}\C(LX,LLY).
$$
Since $\map(-,-)$ is homotopy invariant, we can assume that $LX$, $LLX$, $LLY$ and $LLLY$ are fibrant and cofibrant. Hence, there are two cylinder objects
$$
\xymatrix@R=0.2cm@C=1.1cm{LX\coprod LX \ar[r]^-{i_0\coprod i_1} & \Cyl(LX) \ar[r]^-p & LX & \mbox{\rm and}\\ LLY\coprod LLY \ar[r]^-{i'_0\coprod i'_1} & \Cyl(LLY) \ar[r]^-{p'} & LLY, & }
$$
a left homotopy $H\colon \Cyl(LX)\to LLX$ between $H\circ i_0 =L\ell_X$ and $H\circ i_1 =\ell_{LX}$, and a left homotopy $H'\colon \Cyl(LLY)\to LLLY$ between $H'\circ i'_0 =L\ell_{LY}$ and $H'\circ i'_1 =\ell_{LLY}$ (notice that $H'$ is forced to be a weak equivalence). Let $\tilde{H}^n$ be the functor that sends an object like (\ref{zig-zag}) in $\calL^H_n\C(LX,LLY)$ to
$$
\xymatrix@C=0.7cm{LX & \Cyl(LX) \ar[l]_-{p} \ar[r]^-{H} & LLX & \ar[l]_-{L d_0} L C_1 \dots L C_{n-1} & \ar[l]_-*!/_3pt/{\labelstyle L d_{n-1}\circ H'} \Cyl(LLY) \ar[r]^-{p'} & LLY & \ar[l]_-{id} LLY}
$$
in $\calL^H_{n+4}\C(LX,LLX)$. The diagram
$$
\xymatrix@C=0.7cm{ LX \ar@{=}[d] & LX \ar[l]_-{id }\ar[r]^-{id} \ar[d]_-{i_1} & LX \ar[d]_{\ell_{LX}} & \ar[l]_-{d_0} C_1 \ar@<-4ex>[d]_{\ell_{C_1}} \dots C_{n-1} \ar@<2ex>[d]^{\ell_{C_{n-1}}} & \ar[l]_-{d_{n-1}} LLY \ar[d]_{i'_1} \ar[r]^-{id} & LLY \ar@{=}[d] & \ar[l]_-{id} LLY \ar@{=}[d] \\ LX & \Cyl(LX) \ar[l]_-{p} \ar[r]^-{H} & LLX & \ar[l]_-{L d_0} L C_1 \dots L C_{n-1} & \ar[l]_-*!/_3pt/{\labelstyle L d_{n-1}\circ H'} \Cyl(LLY) \ar[r]^-{p'} & LLY & \ar[l]_-{id} LLY \\ LX \ar@{=}[u] & LX \ar[l]_-{id} \ar[r]^-{L\ell_{X}} \ar[u]^{i_0} & LLX \ar@{=}[u] & \ar[l]_-{L d_0} L C_1 \ar@{=}@<4ex>[u] \dots L C_{n-1} \ar@{=}@<-2ex>[u] & \ar[l]_-*!/_3pt/{\labelstyle L d_{n-1}\circ L\ell_{LY}} LLY \ar[u]^{i'_0} \ar[r]^-{id} & LLY \ar@{=}[u] & \ar[l]_-{id} LLY \ar@{=}[u]}
$$
in $\calL^H_{n+4}\C(LX,LLY)$ defines the zig-zag of natural transformations (\ref{H-zig-zag}) inducing the homotopy equivalence between $id$ and $\calL^H L\circ h$.

The second part of the statement follows from Theorem \ref{thm-nat-transf}, because $\ell$ induces a homotopy $\ell_{LY*}\simeq \ell_X^*\circ \calL^H(L) (X,LY)$, and $\ell_{LY*}$ and $\calL^H L(X,LY)$ are weak homotopy equivalences.
\end{proof}

\begin{proposition}\label{prop}
Let $\M$ be a model category and let $(L,\ell)$ be a homotopy idempotent functor on
$\M$. Then the class of $L$\nobreakdash-equivalences coincides with the simplicial
orthogonal complement of the class of $L$\nobreakdash-local objects.
\end{proposition}

%Part of Proposition \ref{prop} follows directly from \cite[Theorems 3.2.17 and 3.2.18]{Hi03}, but we will %need the full strength of Lemma \ref{lemma-eta} to identify $L$-local objects.

\begin{proof}
We first prove that $L$\nobreakdash-local objects are simplicially orthogonal to $L$\nobreakdash-equi\-valences: Fix an object $LY$ and a morphism $f\colon A\to B$ such that $Lf$ is a weak equivalence. We want to prove that $\map(f,LY)$ is a weak homotopy equivalence. In the commutative diagram
$$
\xymatrix{\map(LB,LY)\ar[r]^{Lf^\ast} \ar[d]_{\ell_B^\ast} &
\map(LA,LY) \ar[d]_{\ell_A^\ast} \\ \map(B,LY)\ar[r]^{f^\ast} &
\map(A,LY)}
$$
the vertical arrows are weak homotopy equivalences by Lemma \ref{lemma-eta} and the top arrow is also a weak homotopy equivalence because $Lf$ is a weak equivalence. Hence, the bottom map has to be a weak homotopy equivalence.

If $f\colon A\to B$ is such that $\map(f,X)$ is a weak homotopy equivalence for each $L$\nobreakdash-local object $X$, using Lemma \ref{lemma-eta} we deduce that $\map(Lf,LA)$ and $\map(Lf,LB)$ are weak homotopy equivalences. Hence, $Lf$ must be a weak equivalence by \cite[Proposition 17.7.6]{Hi03}.%, which means that $LY$ and $f$ are simplicially orthogonal. %also iso in the homotopy category in simplicial model categories \cite[Proposition9.6.9]{Hi03}

Finally, let $X$ be fibrant and such that $\map(f,X)$ is a weak homotopy equivalence for all $L$\nobreakdash-equi\-valences $f$. In particular, $\map(\ell_X,X)$ is a weak homotopy equivalence. On the other hand, $\map(\ell_X,LX)$ is a weak homotopy equivalence by Lemma \ref{lemma-eta}. Hence, $\ell_X\colon X\to LX$ must be a weak equivalence by \cite[Proposition 17.7.6]{Hi03}.
\end{proof}

In what follows, we specialize to combinatorial model categories, i.e., cofibrantly generated model
categories whose underlying category is locally presentable. Since
they have become a standard notion in homotopy theory we refer to
\cite{Du01} or \cite{Ba10} for expositions of the subject. In a left proper combinatorial model category left Bousfield localizations with respect to a set always exist, cf.\ \cite[Theorem~4.7]{Ba10}. The analogue for cellular model categories is proved in \cite[Theorem 4.1.1]{Hi03}.

%Nevertheless, we recall some of their
%basic properties. %, cf.\ \cite[Propostion~2.3]{Du01}.
%If $\M$ is a combinatorial model category, then there exists a regular cardinal $\mu$ satisfying the %following properties:
%\begin{enumerate}
%\item[{\rm 1.}] There exist a cofibrant replacement functor and a fibrant replacement
%functor that preserve $\mu$\nobreakdash-filtered colimits.
%
%\item[{\rm 2.}] A $\mu$\nobreakdash-filtered colimit of weak equivalences is
%a weak equivalence.
%
%\item[{\rm 3.}] There are functorial factorizations preserving
%$\mu$\nobreakdash-smallness.
%\end{enumerate}
%%In this situation, $\M$ is called \emph{$\mu$-combinatorial}.
%
%We also have the following fundamental result about combinatorial model categories.
%
%\begin{theorem}[Smith, {\cite[Theorem~4.7]{Ba10}}]\label{smith}
%Let $\M$ be a left proper combinatorial model category and $\calS$ a set of morphisms in $\M$. Then the %left Bousfield localization $L_\calS\M$ with respect to $\calS$ exists. \qed
%\end{theorem}
%
%In the case that $\M$ is left proper and cellular the analogue of Theorem \ref{smith} is proved in %\cite[Theorem 4.1.1]{Hi03}. If we assume that Vop\v{e}nka's principle holds, then \cite[Theorem~2.3]{RT03} %shows that Theorem \ref{smith} is still valid if we take $\calS$ to be a proper class of morphisms. This %first result of this kind was proved for the category of simplicial sets in \cite{CSS} and was then %extended to every combinatorial simplicial model category in \cite{CC06}.

The next two results correspond to \cite[Lemma 1.2]{CC06} and \cite[Lemma 1.3]{CC06}, but we drop the hypothesis that the model category be simplicial.

\begin{lemma}\label{tech-lem}
Let $\M$ be a combinatorial model category. Then there is
regular cardinal $\mu$ such that, for every class of objects
$\mathcal{D}$ in $\M$, the class of
$\mathcal{D}$\nobreakdash-equivalences $\mathcal{D}^\perp$ is closed under $\mu$\nobreakdash-filtered
colimits.
\end{lemma}

\begin{proof}
Since $\map(-,-)$ is homotopy invariant, we can assume that each object in $\mathcal{D}$ is fibrant. Since we are assuming that $\M$ is combinatorial, there is a regular cardinal $\mu$ such that
weak equivalences are preserved by $\mu$\nobreakdash-filtered colimits and there are cofibrant and fibrant replacement functors that preserve $\mu$\nobreakdash-filtered colimits. Let
$f_i\colon X_i\to Y_i$ be $\mathcal{D}$\nobreakdash-equivalences for all $i\in
I$, where $I$ is a $\mu$\nobreakdash-filtered category. Since we are assuming that cofibrant replacement preserves $\mu$\nobreakdash-filtered colimits, we can assume that $X_i$ and $Y_i$ are cofibrant for all $i\in I$.

We have a commutative diagram
$$
\xymatrix@C=1.6cm{\colim_I X_i \ar[r]^-{\colim_I f_i} & \colim_I Y_i \\ \hocolim_I X_i \ar[r]^-{\hocolim_I f_i} \ar[u] & \hocolim_I Y_i \ar[u] }
$$
where the vertical arrows are weak equivalences since $\mu$\nobreakdash-filtered colimits are homotopy co\-limits, due to the fact that $\mu$\nobreakdash-filtered
colimits of weak equivalences are weak equivalences. To finish the
proof it is enough to prove that the bottom arrow is a
$\mathcal{D}$\nobreakdash-equivalence. But now, for every object
$Z\in\mathcal{D}$ we have a commutative square
$$
\xymatrix@C=1cm{\map(\hocolim_I X_i,Z) \ar[r] \ar[d] & \map(\hocolim_I Y_i,Z) \ar[d] \\ \holim_I \map(X_i,Z) \ar[r] & \holim_I \map(Y_i,Z) }
$$
where the vertical arrows are weak homotopy equivalences by \cite[Theorem
19.4.4]{Hi03}, and the bottom arrow is a weak homotopy equivalence since
every $f_i$ is a $\mathcal{D}$\nobreakdash-equivalence. This proves that
$\hocolim_I f_i$ is a $\mathcal{D}$\nobreakdash-equivalence.
\end{proof}

In the following statement we will need to assume \emph{Vop\v{e}nka's principle}. It is a set\nobreakdash-theore\-tical axiom guaranteeing that every full subcategory of a locally presentable category which is closed under limits is a locally presentable reflective subcategory, i.e., the inclusion has a left adjoint, cf.\ \cite[Theorem~6.6]{AR94}.

\begin{lemma}\label{lem-ho-idemp-is-L_f}
Assume that Vop\v{e}nka's principle holds. Let $\M$ be a
combinatorial model category and let $\mathcal{D}$ be any class of objects in $\M$. Then there is a set of morphisms
$S$ such that the class of $S^{\perp}$\nobreakdash-equivalences, $(S^{\perp})^{\perp}$, is equal to the class
of $\mathcal{D}$\nobreakdash-equivalences, $\mathcal{D}^\perp$. Hence, $(\mathcal{D}^{\perp})^{\perp}$ equals the class of $S$-locals.
\end{lemma}

\begin{proof}
By Lemma \ref{tech-lem}, there is a
regular cardinal $\mu'$ such that, for every class of objects
$\mathcal{E}$ in $\M$, the class of $\mathcal{E}$\nobreakdash-equivalences $\mathcal{E}^\perp$ is closed under $\mu'$\nobreakdash-filtered colimits. On the other hand, $\M$ is $\lambda$\nobreakdash-presentable for some regular cardinal $\lambda$ and so is the category of arrows of $\M$ \cite[Corollary~1.54]{AR94}. Since we are under Vop\v{e}nka's principle, by \cite[Theorem~6.24]{AR94} there exists a regular cardinal $\lambda'$ and a set of $\lambda'$\nobreakdash-presentable $\mathcal{D}$\nobreakdash-equivalences $S'$ such that every morphism in $\mathcal{D}^\perp$ is a $\lambda'$\nobreakdash-filtered colimit of morphisms in $S'$. It then follows that there exists a cardinal $\mu\geq\max\{\lambda', \mu'\}$ and a set of $\mathcal{D}$\nobreakdash-equivalences $S$ such that every morphism in $\mathcal{D}^\perp$ is a $\mu$\nobreakdash-filtered colimit of morphisms in $S$ and $\mathcal{D}^\perp$ is closed under $\mu$\nobreakdash-filtered colimits \cite[Corollary~2.14]{AR94}.

Since every object in $\mathcal{D}$ is $S$\nobreakdash-local, every
$S$\nobreakdash-equivalence is in $\mathcal{D}^{\perp}$. Conversely, every
$g$ in $\mathcal{D}^{\perp}$ is a $\mu$\nobreakdash-filtered colimit of
morphisms in $S$. But now $S\subset (S^\perp)^\perp$ and $(S^\perp)^\perp$ is closed under $\mu'$\nobreakdash-filtered colimits by the first comment in the proof. In particular, $(S^\perp)^\perp$ is also closed under $\mu$\nobreakdash-filtered colimits. This implies that $g$ is in $(S^\perp)^\perp$.
\end{proof}

As a direct consequence of Lemma \ref{lem-ho-idemp-is-L_f}, we obtain an alternative proof of \cite[Theorem~2.1]{CC06} that avoids the assumption of the model category being simplicial. A different proof was given in \cite[Theorem~2.3]{RT03}.
%The left properness is necessary to be able to identify $S$-local objects to fibrant objects in the category $L_S\M$ (or pushouts of H-equiv.-cofibrations are H-equiv).

\begin{corollary}\label{Bousf-combinatorial}
Assume that Vop\v{e}nka's principle holds. Let $\M$ be a left proper
combinatorial model category. Then the left Bousfield localization with respect to any class of morphisms $\calS$ in $\M$ exists.
\end{corollary}

\begin{proof}
By Lemma \ref{lem-ho-idemp-is-L_f}, the class $\calS^\perp$ coincides with the class $D^\perp$ with respect to a set of morphisms $D$. By \cite[Theorem~4.7]{Ba10}, the Bousfield localization $L_D\M$ with respect to $D$ exists in $\M$. Since the $D^\perp$-equivalences coincide with the $\calS^\perp$-equivalences, $L_D\M$ is also the left Bousfield localization with respect to $\calS$.
\end{proof}

As noticed in \cite{CC06}, in general we cannot take $S$ in the conclusion of Lemma
\ref{lem-ho-idemp-is-L_f} to consist of a single morphism. However, it is possible to reduce $S$ to a single morphism if we assume, for instance, that we work in a pointed category. In particular, the next result applies to stable combinatorial model categories.

\begin{corollary}\label{Cor-L_f}
Assume that Vop\v{e}nka's principle holds. Let $\M$ be a pointed combinatorial model category and let $\mathcal{D}$ be any class of objects in $\M$. Then there is a morphism $f$ such that the class of $f$\nobreakdash-equivalences is equal to the class of $\mathcal{D}$\nobreakdash-equivalences.
\end{corollary}

\begin{proof}
Let $S$ be the set of morphisms and $\mu$ the regular cardinal as in the proof of Lemma
\ref{lem-ho-idemp-is-L_f} and let $f=\coprod s$ for all $s\colon A\to B$ in
$S$. It is enough to prove that $S^{h\perp}=f^{h\perp}$. If $X$ is an $S$-local object, then every component in the following product
$$
\prod\map(s,X)\simeq \map\left(\coprod s, X\right) =\map(f,X)
$$
is a weak homotopy equivalence. Hence, $X$ is $f$-local.

Conversely, if $X$ is $f$-local, then $\prod\map(s,X)$ is a weak homotopy equivalence. Since $\M$ is pointed, for each $s$ in $S$ there is a retraction map $r$ such that the composition
$$
\xymatrix{\map(s,X) \ar[r]^-r & \map(\coprod s,X)\simeq\prod\map( s, X) \ar[r] & \map(s,X).}
$$
is the identity. Hence, $\map(s,X)$ is a weak homotopy equivalence for each $s$ in $S$. Thus, $X$ is $S$-local.
\end{proof}

The following result is a direct consequence of Proposition \ref{prop} and Lemma \ref{lem-ho-idemp-is-L_f}.

\begin{theorem}\label{thm-ho-idemp-is-L_f}
Assume that Vop\v{e}nka's principle holds. Let $\M$ be a
combinatorial model category.
If $(L,\ell)$ is any homotopy idempotent functor on $\M$, then
there is a set of morphisms $S$ such that the class of $S$\nobreakdash-local objects coincides with the class of $L$\nobreakdash-local objects. Furthermore, if $\M$ is pointed, then we can take $S$ to consist of a single morphism.
\end{theorem}

\begin{proof}
Let $\mathcal{D}$ be the class of $L$\nobreakdash-local objects. It follows
from Proposition \ref{prop} that the class of
$\mathcal{D}$\nobreakdash-equivalences coincides with the class of
$L$\nobreakdash-equivalences. Then Lemma \ref{lem-ho-idemp-is-L_f} and Corollary \ref{Cor-L_f}
finish the proof.
\end{proof}

%%%%%%%%%%%%%%%%%%%%%%%%%%%%%%%%%%%%%%%%%%%%%%%%%%%%%%

We next extend Theorem \ref{thm-ho-idemp-is-L_f} to any cofibrantly generated model category, in particular to any cellular model category \cite[Definition 12.1.1]{Hi03}.

We remind the reader that a Quillen pair $F\colon\mathcal{N}\rightleftarrows\M:G$ is \emph{homotopically surjective} if, for every fibrant object $X$ in $\M$ and every cofibrant replacement $(G X)^c$ of $G X$, the induced morphism $F (G X)^c \to X$ is a weak equivalence \cite[Definition 3.1]{Du01}.% Equivalently, the compositions of the derived functors $\mathbb{F}\circ \mathbb{G}$ is naturally isomorphic to the identity

\begin{proposition}\label{prop-cofib-gen}
Assume that Vop\v{e}nka's principle holds. Let $F\colon\mathcal{N}\rightleftarrows\M:G$ be a homotopically surjective Quillen pair and let $\mathcal{N}$ be combinatorial. If $(L,\ell)$ is a homotopy idempotent functor on $\M$, then
there is a set of morphisms $S$ in $\M$ such that the class of $S$\nobreakdash-local objects coincides with the class of $L$\nobreakdash-local objects. Furthermore, if $\M$ is pointed, then we can take $S$ to consist of a single morphism.
\end{proposition}

\begin{proof}
Let $\mathcal{D}$ be the class of objects of the form $G X$ with $X$ $L$-local. Notice that they are fibrant because $G$ preserves fibrant objects. By Lemma \ref{lem-ho-idemp-is-L_f}, there is a set of morphisms $S'$ in $\mathcal{N}$ such that the class of $S'$-locals coincide with $(\mathcal{D}^{\perp})^{\perp}$. Let $S=\{F f^c\mid f\in S'\}$. We claim that the $L$-locals coincide with the $S$-locals.

Let $X$ be $L$-local (thus fibrant). By hypothesis, the morphism $F (G X)^c\to X$ is a weak equivalence. By definition, $G X$ is $S'$-local. Hence
$$
\map(f,G X)\simeq \map(f,G(F(G X)^c)^f)\simeq \map(F(f^c), F(G X)^c)
$$
are weak homotopy equivalences for any $f$ in $S'$. In particular, $X\simeq F(G X)^c$ is $S$-local.

Now let $X$ be $S$-local. By definition, $\map(F f^c, X)\simeq \map(f,G X)$ are weak homotopy equivalences for every $f$ in $S'$. Hence $G X$ is $S'$-local, i.e., $GX$ is in $(\mathcal{D}^{\perp})^{\perp}$.

By Proposition \ref{prop}, to prove that $GX$ is $L$-local it is enough to prove that $\map(g,GX)\simeq\map(Fg^c,X)$ are weak equivalences for all $L$-equivalences $g$. Since we have already proved that $GX$ is $D^{\perp}$-local, the proof will be finished if we can show that $g$ is a $\D$-equivalence if and only if $Fg^c$ is an $L$-equivalence. But, by Proposition \ref{prop} again, both conditions are equivalent to the fact that $\map(g,GY)\simeq\map(Fg^c,Y)$ is a weak equivalence for all $L$-local objects $Y$.
\end{proof}

The following result generalizes \cite[Theorem~2.3]{CC06} to cofibrantly generated model categories that are not necessarily locally presentable nor simplicial. It also gives a positive answer to a question by Farjoun in \cite{Far96} for a broad family of model categories.

\begin{corollary}\label{cor-cofib-gen1}
Assume that Vop\v{e}nka's principle holds. Let $\M$ be a
cofibrantly generated model category. If $(L,\ell)$ is a homotopy idempotent functor on $\M$, then there is a set of morphisms $S$ such that the class of $S$\nobreakdash-local objects coincides with the class of $L$\nobreakdash-local objects. Furthermore, if $\M$ is pointed, then we can take $S$ to consist of a single morphism.
\end{corollary}

\begin{proof}
Since we are assuming Vop\v{e}nka's principle, \cite[Theorem~1.1]{Ra09} implies that there is a Quillen
equivalence (in particular homotopically surjective) $\mathcal{N}\rightleftarrows\M$ where
$\mathcal{N}$ is combinatorial. Hence, the result follows from Proposition \ref{prop-cofib-gen}.
\end{proof}

%In fact, the result holds for a larger class of model categories as noticed in \cite{Ro09}.
%
%Notice that any model category in which every object is cofibrant is left proper, see \cite[Corollary %13.1.3]{Hi03}.

The cofibrantly generated condition in Corollary \ref{cor-cofib-gen1} is necessary. In \cite{Ch05} an example is given of a left Bousfield localization with respect to a class of morphisms in a (non cofibrantly generated) model category that cannot be a left Bousfield localization with respect to any set.

%I don't have a proof of that!!
%\begin{corollary}\label{right-Bousf-combinatorial}
%Assume that Vop\v{e}nka's principle holds. Let $\M$ be a left proper
%cofibrantly generated model category. Then the left Bousfield localization with respect to any class of %morphisms in $\M$ exists. \qed
%\end{corollary}

%%%%%%%%%%%%%%%%%%%%%%%%%%%%%%%%%%%%%%%%%%%%%

%So far, in this section, we have worked with \emph{coaugmented} homotopy idempotent functors. The results %we have presented are also valid for the \emph{augmented} case.

We next state the analogues of the main results in this section but for \emph{augmented} homotopy idempotent functors. We omit the proofs since they are easily reproduced following the proofs for the coaugmented case. An \emph{augmented homotopy idempotent} functor in a model category $\M$ is a functor $C\colon\M\to\M$ together with a natural transformation $\varepsilon\colon C\to 1$ such that $C$ sends weak equivalences to weak equivalences and the natural
morphisms $\varepsilon_{CX},C\varepsilon_X\colon CCX\to CX$ are equal in the homotopy category $\Ho(\M)$ and both are weak equivalences for every object $X$ in~$\M$. The following result generalizes \cite[Theorem~1.4]{Ch07} to combinatorial model categories not necessarily simplicial.

%I can't prove it for cofibrantly generated!!!
\begin{corollary}\label{cell2}
Assume that Vop\v{e}nka's principle holds. Let $\M$ be a right proper
combinatorial model category. Then the right Bousfield localization with respect to any class of objects in $\M$ exists. \qed
\end{corollary}

The following result generalizes \cite[Theorem~2.1]{Ch07} to cofibrantly generated model categories not necessarily locally presentable nor simplicial.

\begin{theorem}\label{cell1}
Assume that Vop\v{e}nka's principle holds. Let $\M$ be a
cofibrantly generated model category. If $(C,\varepsilon)$ is a homotopy augmented idempotent functor on $\M$, then there is a set of objects $D$ such that the class of $D$\nobreakdash-cellular equivalences coincides with the class of $C$\nobreakdash-cellular equivalences. Furthermore, if $\M$ is pointed, then we can take $D$ to consist of a single object. \qed
\end{theorem}
%To see that we cannot reduce to a single object in general: In $\sSet\times \sSet$, take the objects $A=(\empty,*)$ and $B=(*,\empty)$. Then $f:(\empty,b)\to(a,*)$ is a $\{A,B\}$-equivalence iff $b$ is contractible and $a$ is empty. Observe that $A\coprod B=(*,*)$ and then $(*,*)$-equivalences are weak equivalences. In general, if it would exist $(c,d)$ with $(c,d)$-equiv=$\{A,B\}$-equiv, then c must be empty. Hence, $f$ would be an $(c,d)$-equiv for ALL $a$ and $b$ contractible, which is a contradiction.

\providecommand{\bysame}{\leavevmode\hbox to3em{\hrulefill}\thinspace}
%\providecommand{\MR}{\relax\ifhmode\unskip\space\fi MR }
\providecommand{\MRhref}[2]{%
  \href{http://www.ams.org/mathscinet-getitem?mr=#1}{#2}
}
\providecommand{\href}[2]{#2}


\begin{thebibliography}{DHKS04}

\bibitem[Ada75]{Ad75}
J.~F. Adams, \emph{Localisation and completion}, Department of Mathematics,
  University of Chicago, Chicago, Ill., 1975.

\bibitem[AR94]{AR94}
J.~Ad{\' a}mek and J.~Rosick{\'y}, \emph{Locally presentable and accessible
  categories}, London Mathematical Society Lecture Note Series, vol. 189,
  Cambridge University Press, Cambridge, 1994.

\bibitem[Bar10]{Ba10}
C. Barwick, \emph{On left and right model categories and left and right
  {B}ousfield localizations}, Homology, Homotopy Appl. \textbf{12} (2010),
  no.~2, 245--320.

\bibitem[Bor94]{Bo94II}
F. Borceux, \emph{Handbook of categorical algebra. 2}, Encyclopedia of
  Mathematics and its Applications, vol.~51, Cambridge University Press,
  Cambridge, 1994, Categories and structures.

\bibitem[CC06]{CC06}
C. Casacuberta and B. Chorny, \emph{The orthogonal subcategory problem
  in homotopy theory}, An alpine anthology of homotopy theory, Contemp. Math.,
  vol. 399, Amer. Math. Soc., Providence, RI, 2006, 41--53.

%\bibitem[CGR14]{CGR}
%C. Casacuberta, J.~J. Guti{\'e}rrez, and Ji{\v{r}}{\'{\i}} Rosick{\'y},
%  \emph{Are all localizing subcategories of stable homotopy categories
%  coreflective?}, Adv. Math. \textbf{252} (2014), 158--184.

\bibitem[CRT]{CaRaTo}
C.~Casacuberta, O.~Ravent{\'o}s, and A.~Tonks, \emph{Comparing localizations
  across adjunctions}, arXiv:1404.7340.

\bibitem[CSS05]{CSS}
C. Casacuberta, D. Scevenels, and J.~H. Smith, \emph{Implications of
  large-cardinal principles in homotopical localization}, Adv. Math.
  \textbf{197} (2005), no.~1, 120--139.

\bibitem[Cho05]{Ch05}
B. Chorny, \emph{Localization with respect to a class of maps. II. Equivariant cellularization and its application}, Israel J. Math. \textbf{147} (2005), 141--155.

\bibitem[Cho07]{Ch07}
\bysame, \emph{Abstract cellularization as a cellularization with respect to a set of objects}, Categories in algebra, geometry and mathematical physics, Contemp. Math., vol. 431, Amer. Math. Soc., Providence, RI, 2007, 165--170.

\bibitem[DHKS04]{DHKS04}
W.~G. Dwyer, P.~S. Hirschhorn, D.~M. Kan, and J.~H. Smith, \emph{Homotopy limit
  functors on model categories and homotopical categories}, Mathematical
  Surveys and Monographs, vol. 113, American Mathematical Society, Providence,
  RI, 2004.

\bibitem[DK80a]{DK80b}
W.~G. Dwyer and D.~M. Kan, \emph{Calculating simplicial localizations}, J. Pure
  Appl. Algebra \textbf{18} (1980), no.~1, 17--35.

\bibitem[DK80b]{DK80c}
\bysame, \emph{Function complexes in homotopical algebra}, Topology \textbf{19}
  (1980), no.~4, 427--440.

\bibitem[DK80c]{DK80a}
\bysame, \emph{Simplicial localizations of categories}, J. Pure Appl. Algebra
  \textbf{17} (1980), no.~3, 267--284.

\bibitem[Dug01]{Du01}
D.~Dugger, \emph{Combinatorial model categories have presentations}, Adv. Math.
  \textbf{164} (2001), no.~1, 177--201.

%\bibitem[GZ67]{GZ67}
%P.~Gabriel and M.~Zisman, \emph{Calculus of fractions and homotopy theory},
%  Ergebnisse der Mathematik und ihrer Grenzgebiete, Band 35, Springer-Verlag
%  New York, Inc., New York, 1967.

\bibitem[Far96]{Far96}
E.~D. Farjoun, \emph{Cellular spaces, null spaces and homotopy localization}, Lecture Notes in Mathematics, 1622, Springer-Verlag, Berlin, 1996.

\bibitem[Gra74]{Gr74}
J.~W. Gray, \emph{Formal category theory: adjointness for 2-categories}, Lecture Notes in Math., vol. 391, Springer-Verlag, Berlin-New York, 1974.

\bibitem[Hir03]{Hi03}
P.~S. Hirschhorn, \emph{Model categories and their localizations}, Mathematical
  Surveys and Monographs, vol.~99, American Mathematical Society, Providence,
  RI, 2003.

%\bibitem[Mac98]{Mac98}
%S.~Mac~Lane, \emph{Categories for the working mathematician}, 2nd ed., Graduate Texts in Mathematics, vol. %5, Springer-Verlag, New York, 1998.

\bibitem[Qui73]{Qu73}
D.~Quillen, \emph{Higher algebraic {$K$}-theory. {I}}, Algebraic {$K$}-theory,
  {I}: {H}igher {$K$}-theories ({P}roc. {C}onf., {B}attelle {M}emorial {I}nst.,
  {S}eattle, {W}ash., 1972), Lecture Notes in Math., vol. 341, Springer, Berlin, 1973.

\bibitem[Rap09]{Ra09}
G. Raptis, \emph{On the cofibrant generation of model categories}, J.
  Homotopy Relat. Struct. \textbf{4} (2009), no.~1, 245--253.

\bibitem[RT03]{RT03}
J. Rosick\'{y} and W. Tholen, \emph{Left-determined model categories and universal homotopy theories}, Trans. Amer. Math. Soc. \textbf{355} (2003), no.~9, 3611--3623.

\end{thebibliography}
\end{document}